\documentclass[preprint]{imsart}

\usepackage{amssymb,amsbsy,amsmath,amscd,amsthm,amsfonts,MnSymbol}
\usepackage{cite}
\usepackage[utf8]{inputenc}
\usepackage{enumerate,hyperref}
\usepackage{bbm, dsfont}
\usepackage{graphicx}

\hypersetup{
hidelinks
}

\newtheorem{theorem}{Theorem}

\newtheorem{lemma}{Lemma}
\newtheorem{remark}{Remark}

\newtheorem{corollary}{Corollary}
\newtheorem{assumption}{Assumption}

\newcommand{\R}{\mathbb R}
\newcommand{\PP}{\mathbb P}
\newcommand{\E}{\mathbb E}
\newcommand{\K}{\mathcal K_d}

\newcommand{\Sd}{\mathbb S^{d-1}}
\newcommand{\dH}{\textsf{d}_{\textsf H}}

\newcommand*\diff{\mathop{}\!\mathrm{d}}
\newcommand{\DS}{\displaystyle}

\begin{document}

\begin{frontmatter}

\title{Estimation of convex supports from noisy measurements}
\runtitle{Estimation of convex supports from noisy measurements}

\begin{aug}
  \author{Victor-Emmanuel Brunel\ead[label=e1]{vebrunel@mit.edu}}
  \address{Department of Mathematics, Massachusetts Institute of Technology\\ 
           \printead{e1}}

  \author{Jason M. Klusowski\ead[label=e2]{jason.klusowski@yale.edu}}
  \and
  \author{Dana Yang\ead[label=e3]{xiaoqian.yang@yale.edu}}

  \address{Department of Statistics and Data Science, Yale University\\
          \printead{e2,e3}}


  \runauthor{Brunel, Klusowski, Yang}

\end{aug}

\begin{abstract}
A popular class of problem in statistics deals with estimating the support of a density from $ n $ observations drawn at random from a $ d $-dimensional distribution. The one-dimensional case reduces to estimating the end points of a univariate density. In practice, an experimenter may only have access to a noisy version of the original data. Therefore, a more realistic model allows for the observations to be contaminated with additive noise.

In this paper, we consider estimation of convex bodies when the additive noise is distributed according to a multivariate Gaussian distribution, even though our techniques could easily be adapted to other noise distributions. Unlike standard methods in deconvolution that are implemented by thresholding a kernel density estimate, our method avoids tuning parameters and Fourier transforms altogether. 
We show that our estimator, computable in $(O(\ln n))^{(d-1)/2}$ time, converges at a rate of $ O_d(\log\log n/\sqrt{\log n}) $ in Hausdorff distance, in accordance with the polylogarithmic rates encountered in Gaussian deconvolution problems. Part of our analysis also involves the optimality of the proposed estimator. We provide a lower bound 
for the minimax rate of estimation in Hausdorff distance that is $ \Omega_d(1/\log^2 n) $.
\end{abstract}


\end{frontmatter}

\section{Preliminaries}

\subsection{Introduction}

The problem of estimating the support of a distribution, given i.i.d. samples, poses both statistical and computational questions. When the support of the distribution is known to be convex, geometric methods have been borrowed from stochastic and convex geometry with the use of random polytopes since the seminal works \cite{Renyi1963,Renyi1964}. When the distribution of the samples is uniform on a convex body, estimation in a minimax setup has been tackled in \cite{KTbook} (see also the references therein). There, the natural estimator defined as the convex hull of the samples (which is referred to as \emph{random polytope} in the stochastic geometry literature) is shown to attain the minimax rate of convergence on the class of convex bodies, under the Nikodym metric. 


When the samples are still supported on a convex body but their distribution is no longer uniform, \cite{Brunel2017-2015} studies the performance of the random polytope as an estimator of the convex support under the Nikodym metric, whereas \cite{BrunelHausdorff2017} focuses on the Hausdorff metric. In the latter, computational issues are addressed in higher dimensions. Namely, determining the list of vertices of the convex hull of $n$ points in dimension $d\geq 2$ is very expensive, namely, exponential in $d\log n$ (see \cite{Chazelle1993}). In \cite{BrunelHausdorff2017}, a randomized algorithm produces an approximation of the random polytope that achieves a trade-off between computational cost and statistical accuracy. The approximation is given in terms of a membership oracle, which is a very desirable feature for the computation/approximation of a convex body.

Both works \cite{Brunel2017-2015,BrunelHausdorff2017} assume that one has access to direct samples. Here, we are interested in the case when samples are contaminated, more specifically, subject to measurement errors. In \cite{Marteau2015}, a closely related problem is studied, where two independent contaminated samples are observed, and one wants to estimate the set where $f-g$ is positive, where $f$ and $g$ are the respective densities of the two samples. In that work, the contamination is modeled as an additive noise with known distribution, and some techniques borrowed from inverse problems are used. The main drawback is that the estimator is not tractable and it only gives a theoretical benchmark for minimax estimation.

Goldenshluger and Tsybakov \cite{Goldenshluger2004} study the problem of estimating the endpoint of a univariate distribution, given samples contaminated with additive noise. Their analysis suggests that their estimator is optimal in a minimax sense and its computation is straightforward. In our work, we first extend their result, which then we lift to a higher dimensional setup: that of estimating the convex support of a uniform distribution, given samples that are contaminated with additive Gaussian noise. Our method relies on projecting the data points along a finite collection of unit vectors. Unlike in \cite{Marteau2015}, we give an explicit form for our estimator. In addition, our estimator is tractable when the ambient dimension is not too large. If the dimension is too high, the number of steps required to compute a membership oracle for our estimator becomes exponentially large in the dimension: Namely, of order $(O(\ln n))^{(d-1)/2}$.

\subsection{Notation}

In this work, $d\geq 2$ is a fixed integer standing for the dimension of the ambient Euclidean space $\R^d$. The Euclidean ball with center $a\in\R^d$ and radius $r\geq 0$ is denoted by $B_d(a,r)$. The unit sphere in $\R^d$ is denoted by $\Sd$ and $\kappa_d$ stands for the volume of the unit Euclidean ball. 

We refer to convex and compact sets with nonempty interior in $\R^d$ as convex bodies. The collection of all convex bodies in $\R^d$ is denoted by $\K$. Let $\sigma^2>0$ and $n\geq 1$. If $X_1,\ldots,X_n$ are i.i.d. random uniform points in a convex body $G$ and $\varepsilon_1,\ldots,\varepsilon_n$ are i.i.d. $d$-dimensional centered Gaussian random vectors with covariance matrix $\sigma^2 I$, where $I$ is the $d\times d$ identity matrix, independent of the $X_j$'s, we denote by $\PP_G$ the joint distribution of $X_1+\varepsilon_1,\ldots,X_n+\varepsilon_n$ and by $\E_G$ the corresponding expectation operator (we omit the dependency on $n$ and $\sigma^2$ for simplicity).

The support function of a convex set $G\subseteq\R^d$ is defined as $\DS h_G(u)=\sup_{x\in G}\langle u,x\rangle, u\in\R^d$, where $\langle\cdot,\cdot\rangle$ is the canonical scalar product in $\R^d$: It is the largest signed distance between the origin and a supporting hyperplane of $G$ orthogonal to $u$. 

The Hausdorff distance between two sets $A,B\subseteq\R^d$ is $\DS \dH(A,B)=\inf\{\varepsilon>0 : G_1\subseteq G_2+\varepsilon B_d(0,1) \mbox{ and } G_2\subseteq G_1+\varepsilon B_d(0,1)\}$. If $A$ and $B$ are convex bodies, it can be written in terms of their support functions: $\DS \dH(A,B)=\sup_{u\in\Sd}\left|h_A(u)-h_B(u)\right|$. 

For $ f $ in $ L^1(\mathbb{R}^d) $, let $\DS \mathcal{F}[f](t) = \int_{\mathbb{R}^d}e^{i\langle t, x \rangle}f(x)dx$ denote the Fourier transform of $f$.

The total variation distance between two distributions $ P $ and $ Q $ having densities $ p $ and $ q $ with respect to a dominating measure $ \mu $ is defined by $ \text{TV}(P, Q) = \frac{1}{2}\int |p - q|d\mu $.

The Lebesgue measure of a measurable, bounded set $ A $ in $ \mathbb{R}^d $ is denoted by $ |A| $. For a vector $ x = (x_1, x_2, \dots, x_d)\in\R^d $, we define $ \|x\|_p = \left(\sum_{i=1}^d |x_i|^p \right)^{1/p} $ for $ p \geq 1 $ and $ \|x\|_{\infty} = \sup_{1 \leq i\leq d} |x_i| $. For a function, $ f $ defined on a set $ A $, let $ \|f\|_{\infty} = \sup_{x\in A}|f(x)| $. The Nikodym distance between two measurable, bounded sets $ A $ and $ B $ is defined by $ \mathsf{d}_{\Delta}(A, B) = |A \Delta B| $.

We use standard big-$O$ notations: For any positive sequences $\{a_n\}$ and $\{b_n\}$, $a_n=O(b_n)$ or $a_n \lesssim b_n$ if $a_n \leq C b_n$ for some absolute constant $C>0$,
$a_n=o(b_n)$ or $a_n \ll b_n$ if $\lim a_n/b_n = 0$. Finally, we write $a_n \asymp b_n$ when both $a_n\gtrsim b_n$ and $a_n\lesssim b_n$ hold.
Furthermore, the subscript in $a_n=O_{r}(b_n)$ means $a_n \leq C_r b_n$ for some constant $C_r$ depending on the parameter $r$ only. We write $ a_n \propto b_n $ when $ a_n = Cb_n $ for some absolute constant $ C $. We let $ \phi_{\sigma} $ denote the Gaussian density with mean zero and variance $ \sigma^2 $, i.e., $ \phi_{\sigma}(x) = \frac{1}{\sqrt{2\pi}\sigma}e^{-x^2/(2\sigma^2)} $ for all $ x \in \R $.

\subsection{Model and outline}

In what follows, we consider the problem of estimating a convex body from noisy observations. More formally, suppose we have access to independent observations 
\begin{equation} \label{Model}
 Y_j = X_j + \varepsilon_j, \quad j=1, \dots, n,
\end{equation}
where $X_1,\ldots,X_n$ are i.i.d. uniform random points in an unknown convex body $G$ and $\varepsilon_1,\ldots,\varepsilon_n$ are i.i.d. Gaussian random vectors with zero mean and covariance matrix $\sigma^2 I$, independent of $X_1,\ldots,X_n$. In the sequel, we assume that $\sigma^2$ is a fixed and known positive number. The goal is to estimate $ G $ using $ Y_1,\dots,Y_n $. This can be seen as an inverse problem: the object of interest is a special feature (here, the support) of a density that is observed up to a convolution with a Gaussian distribution. Our approach will not use the path of inverse problems, but instead, will be essentially based on geometric arguments.

The error of an estimator $\hat G_n$ of $G$ is defined as $\E_G\left[\dH(\hat G_n,G)\right]$. Let $\mathcal C\subseteq\K$ be a subclass of the class of all convex bodies in $\R^d$. The risk of an estimator $\hat G_n$ on the class $\mathcal C$ is $\DS \sup_{G\in\mathcal C} \E_G\left[\dH(\hat G_n,G)\right]$ and the minimax risk on $\mathcal C$ is defined as
$$\mathcal R_n(\mathcal C)=\inf_{\hat G}\sup_{G\in\mathcal C}\E_G\left[\dH(\hat G,G)\right],$$
where the infimum is taken over all estimators $\hat G$ based on $Y_1,\ldots,Y_n$. The minimax rate on the class $\mathcal C$ is the speed at which $\mathcal R_n(\mathcal C)$ goes to zero.

Our strategy for estimating $G$ avoids standard methods from inverse problems that would require Fourier transforms and tuning parameters. To give intuition for our procedure, first observe that a convex set can be represented in terms of its support function via
\begin{equation*}
G = \{ x\in\mathbb{R}^d: \langle u, x \rangle \leq h_G(u)\; \text{for all} \; u\in\Sd \}.
\end{equation*}
If we can find a suitable way of estimating $ h_G $, say by $ \hat{h}_n $, then there is hope that an estimator of the form
\begin{equation*}
\hat{G}_n = \{ x\in\mathbb{R}^d: \langle u, x \rangle \leq \hat{h}_n(u)\; \text{for all} \; u\in\Sd \}
\end{equation*}
will perform well. This is the core idea of our procedure: We project the data points $Y_1,\ldots,Y_n$ along unit vectors and for all such $u\in\Sd$, we estimate the endpoint of the distribution of $\langle u,X_1\rangle$ given the one dimensional sample $\langle u,Y_1\rangle,\ldots,\langle u,Y_n\rangle$. 

Section \ref{Section1D} is devoted to the study of the one dimensional case, where we extend the results proven in \cite{Goldenshluger2004}. The one-dimensional case reduces to estimating the end-point of a univariate density. This problem has been extensively studied in the noiseless case \cite{Hall1982,Mason1989} and more recently as an inverse problem \cite{Hall2002,Goldenshluger2004}. In \cite{Goldenshluger2004}, it is assumed that the density of the (one-dimensional) $X_j$'s is \emph{exactly} equal to a polynomial in a neighborhood of the endpoint of the support. We extend their results to the case when the distribution function is only bounded by two polynomials whose degrees may differ, in the vicinity of the endpoint.

In Section \ref{SectiondD}, we use these one dimensional results in order to define our estimator of the support $G$ if the $X_j$'s and to bound its risk on a certain subclass of $\K$. We show that our estimator nearly attains the minimax rate on that class, up to logarithmic factors. 

Intermediate lemmas and proofs of corollaries are deferred to Section \ref{SectionProofs}.

\section{Estimation of the endpoint of a distribution with contaminated samples}\label{Section1D}

Let $\varepsilon_1,\ldots,\varepsilon_n$ be i.i.d. centered Gaussian random variables. Then, the maximum $\max_{1\leq j \leq n}\varepsilon_j$ concentrates around $\sqrt{2\sigma^2\ln n}$, where $\sigma^2=\E[\varepsilon_1^2]$. Our first result shows the same remains true if one adds i.i.d. nonpositive random variables to the $\varepsilon_j$'s, as long as their cumulative distribution function increases polynomially near zero. As a byproduct, one can estimate the endpoint of a distribution with polynomial decay near its boundary by substracting a deterministic bias from the maximum of the observations. In the sequel, set $b_n=\sqrt{2\sigma^2\ln n}$. 

\begin{theorem} \label{Theorem1}

Let $X$ be a random variable with cumulative distribution function $F$ and $\varepsilon$ be a centered Gaussian random variable with variance $\sigma^2>0$, independent of $X$. Let $Y=X+\varepsilon$ and consider a sequence $Y_1,Y_2,\ldots$ of independent copies of $Y$ and define $M_n=\max \{Y_1,\ldots,Y_n\}$, for all $n\geq 1$. Assume that there exist real numbers $\theta_F\in\R$, $\alpha\geq\beta\geq 0$, $r>0$ and $L>0$ such that the following is true: 
\begin{equation*}
	L^{-1}t^\alpha \leq 1-F(\theta_F-t)\leq Lt^{\beta}, \quad \forall t\in [0,r].
\end{equation*}
Then, there exist $n_0\geq 1$ and $c_0,c_1,c_2>0$ that depend on $\alpha,\beta$, $L$, $r$ and $\sigma^2$ only, such that for all $n\geq  n_0$ and $t>0$,
\begin{equation*}
	\PP\left[|M_n-b_n-\theta_F|>\frac{t+c_0\ln \ln n}{b_n}\right]\leq c_1e^{-\frac{t}{2\sigma^2}}+e^{-c_2 n}.
\end{equation*}
\end{theorem}

The expressions of $n_0$ and of the constants $c_1$ and $c_2$ can be easily deduced from the proof of the theorem.

\begin{proof}[Proof of Theorem \ref{Theorem1}]

Denote by $G$ the cumulative distribution function of $Y_1-\theta_F$. We use the following lemma, which we prove in Section \ref{SectionProofLemmas}. 
\begin{lemma} \label{Lemmacdf}
	There exist two positive constants $c$ and $C$ that depend only on $r, L$ and $\alpha$, such that for all $x\geq \sigma^2/r$,
\begin{equation*}
	\frac{ce^{-\frac{x^2}{2\sigma^2}}}{x^{\alpha+1}}\leq 1-G(x)\leq \frac{Ce^{-\frac{x^2}{2\sigma^2}}}{x^{\beta+1}}.
\end{equation*}
\end{lemma}

Let $x$ be a positive number and $n$ be a positive integer. Write that
\begin{equation} \label{Proof1_6}
	\PP\left[|M_n-\theta_F-b_n|>x\right] = 1-G(b_n+x)^n+G(b_n-x)^n.
\end{equation}

Let us first bound from below $G(b_n+x)^n$. Assume that $n$ is sufficiently large so that $b_n\geq r/\sigma^2$. By Lemma \ref{Lemmacdf},
\begin{align}
	G(b_n+x) & \geq 1-\frac{Ce^{-\frac{(b_n+x)^2}{2\sigma^2}}}{(b_n+x)^{\beta+1}} \geq 1-\frac{Ce^{-\frac{(b_n+x)^2}{2\sigma^2}}}{b_n^{\beta+1}} \nonumber \\ 
	& = 1-C\exp\left(-\frac{x^2}{2\sigma^2}-\frac{xb_n}{\sigma^2}-\frac{b_n^2}{2\sigma^2}-(\beta+1)\ln b_n\right) \label{Casa=b1} \\
	& \geq 1-C\exp\left(-\frac{xb_n}{\sigma^2}-\frac{b_n^2}{2\sigma^2}\right) \nonumber \\
	\label{Proof1_7} & = 1-\frac{C}{n}\exp\left(-\frac{xb_n}{\sigma^2}\right),
\end{align}
as long as $n$ is large enough so $\ln b_n\geq 0$. 

Note that for all $u\in[0,1/2]$, $\displaystyle{1-u\geq e^{-2(\ln 2)u}\geq 1-2(\ln 2)u}$. Hence, if $n$ is large enough, \eqref{Proof1_7} implies
\begin{equation} 
	\label{Proof1_12}G(b_n+x)^n \geq 1-2(\ln 2)Ce^{-\frac{xb_n}{2\sigma^2}}.
\end{equation}

Let us now bound from above $G(b_n-x)^n$. First, if $x\leq b_n-r/\sigma^2$, Lemma \ref{Lemmacdf} yields
\begin{align}
	G(b_n-x) & \leq 1-\frac{ce^{-\frac{(b_n-x)^2}{2\sigma^2}}}{(b_n-x)^{\alpha+1}} \leq 1-\frac{c}{b_n^{\alpha+1}}\exp\left(-\frac{x^2}{2\sigma^2}+\frac{xb_n}{\sigma^2}-\frac{b_n^2}{2\sigma^2}\right) \nonumber \\ 
	& \leq 1-c\exp\left(\frac{x b_n}{2\sigma^2}-\frac{b_n^2}{2\sigma^2}-(\alpha+1)\ln b_n\right) \label{casa=b2} \\
	\label{Proof1_11'} & = 1-\frac{ce^{B_1}}{n}\exp\left(\frac{x b_n}{2\sigma^2}-\frac{\alpha+1}{2}\ln \ln n\right),
\end{align}
where $B_1=(1/2)(\alpha+1)\ln(2\sigma^2)$. 
Together with the inequalities $1-u\leq e^{-u}\leq 1/u, \forall u>0$, \eqref{Proof1_11'} implies
\begin{equation} \label{Proof1_13}
	G(b_n-x)^n \leq c^{-1}e^{-B_1}e^{-\frac{xb_n}{2\sigma^2}+\frac{\alpha+1}{2}\ln \ln n}.
\end{equation}
Now, if $x>b_n-r/\sigma^2$, one can simply bound
\begin{align} 
	G(b_n-x)^n & \leq G(r/\sigma^2)^n \nonumber \\
	\label{Proof1_14} & \leq e^{-c_2n},
\end{align}
using Lemma \ref{Lemmacdf}, with $\displaystyle{c_2=-\ln\left(1-\frac{c\sigma^{2\alpha+2}e^{-\frac{r^2}{2\sigma^6}}}{r^{\alpha+1}}\right)}$.
Finally, combining \eqref{Proof1_13} and \eqref{Proof1_14} yields 
\begin{equation} \label{Proof1_15}
	G(b_n-x)^n \leq c^{-1}e^{-B_1}e^{-\frac{xb_n}{2\sigma^2}+\frac{\alpha+1}{2}\ln \ln n}+e^{-c_2n},
\end{equation}
for all positive numbers $x$. Now, plugging \eqref{Proof1_12} and \eqref{Proof1_15} into \eqref{Proof1_6} yields
\begin{equation} \label{Proof1_16}
	\PP\left[|M_n-\theta_F-b_n|>x\right] \leq c_1e^{-\frac{xb_n}{2\sigma^2}+\frac{\alpha+1}{2}\ln \ln n}+e^{-c_2n},
\end{equation}
where $c_1=2(\ln 2)C+c^{-1}e^{-B_1}.$
Taking $x$ of the form $\DS \frac{t+c_0\ln\ln n}{b_n}$ for $t\geq 0$ and $c_0=(\alpha+1)\sigma^2$ yields Theorem \ref{Theorem1}. 

\end{proof}

When $\alpha$ and $\beta$ are equal and known, it is possible to account for the deterministic bias at a higher order and get a more accurate estimate of $\theta_F$.

\begin{theorem} \label{Theorem1BIS}

Let assumptions of Theorem \ref{Theorem1} hold with $\alpha=\beta$. Set $\DS \tilde b_n=\sqrt{2\sigma^2 \ln n}\left(1-\frac{(\alpha+1)\ln\ln n}{4\ln n}\right)$.
Then, there exist $n_0\geq 1$ and $c_1,c_2>0$ that depend on $\alpha$, $L$ and $r$ only, such that for all $n\geq  n_0$ and $t>0$,
\begin{equation*}
	\PP\left[|M_n-\tilde b_n-\theta_F|>\frac{t}{\tilde b_n}\right]\leq c_1e^{-\frac{t}{2\sigma^2}}+e^{-c_2 n}.
\end{equation*}
\end{theorem}

\begin{proof}[Proof of Theorem \ref{Theorem1BIS}]

The proof of Theorem \ref{Theorem1BIS} follows the same lines as that of Theorem \ref{Theorem1}, where $b_n$ is replaced with $\tilde b_n$. The main modification occurs in \eqref{Casa=b1} and \eqref{casa=b2}, where we note that $\DS\ln n-B \leq \frac{\tilde b_n^2}{2\sigma^2}+(\alpha+1)\ln \tilde b_n\leq \ln n+B$, for some positive constant $B$. 

\end{proof}

In Theorem \ref{Theorem1}, $\theta_F$ is the endpoint of the distribution of the $X_j$'s. When $\theta_F$ is unknown, it can be estimated using $\hat\theta_n:=M_n-b_n$ (or $\tilde\theta_n:=M_n-\tilde b_n$ if $\alpha=\beta$ is known). Theorems \ref{Theorem1} and \ref{Theorem1BIS} show that $\hat\theta_n$ and $\tilde \theta_n$ are consistent estimators of $\theta_F$, but that they concentrate very slowly around $\theta_F$, at a polylogarithmic rate. We actually show that this rate is optimal (up to a sublogarithmic factor in the case of $\hat \theta_n$) in a minimax sense.

For every collection of parameters $\alpha\geq\beta\geq 0$, $r>0$ and $L>0$, let $\mathcal F(\alpha,\beta,r,L)$ the class of all cumulative distribution functions $F$ satisfying $\DS L^{-1}t^\alpha \leq 1-F(\theta_F-t)\leq Lt^{\beta}, \forall t\in [0,r]$.

The following result is a direct consequence of Theorem \ref{Theorem1}.

\begin{corollary} \label{TheoremMinimax}
	For all $\alpha\geq\beta\geq 0$, $r>0$ and $L>0$,
	$$\inf_{\hat T_n}\sup_{F\in \mathcal F(\alpha,\beta,r,L)}\E\left[|\hat T_n-\theta_F|\right] \lesssim \begin{cases}  \frac{\ln\ln n}{\sqrt{\ln n}} \mbox{ if } \alpha>\beta, \\ \frac{1}{\sqrt{\ln n}} \mbox{ if } \alpha=\beta, \end{cases}$$
	where the infimum is taken over all estimators $\hat T_n$. All the constants depend only on the parameters $\alpha,\beta,r,L$ and $\sigma^2$.
\end{corollary}

Theorem 2 in \cite{Goldenshluger2004} suggests that the upper bound in Corollary \ref{TheoremMinimax} is optimal, up to a sublogarithmic factor. However, their result is only for a modified version of the model and hence does not show a lower bound that matches their upper bound.

As a conclusion, these results suggest that in the presence of Gaussian errors, the endpoint $\theta_F$ of the distribution of the contaminated data can only be estimated at a polylogarithmic rate, in a minimax sense. In the next section, we prove a lower bound in a multivariate setup, whose rate is polylogarithmic in the sample size.

\section{Application to convex support estimation from noisy data} \label{SectiondD}

In this section, we apply Theorem \ref{Theorem1} to the problem of estimating a convex body from noisy observations of independent uniform random points. Let $G$ be a convex body in $\R^d$ and let $X$ be uniformly distributed in $G$. Let $\varepsilon$ be a $d$-dimensional centered Gaussian random variable with covariance matrix $\sigma^2I$, where $\sigma^2$ is a known positive number and $I$ is the $d\times d$ identity matrix. Let $Y=X+\varepsilon$ and assume that a sample $Y_1,\ldots,Y_n$ of $n$ independent copies of $Y$ is available to estimate $G$. 

Our estimation scheme consists in reducing the $d$-dimensional estimation problem to a 1-dimensional one, based on the following observation. Let $u\in\Sd$. Then, $\langle u,Y\rangle = \langle u,X\rangle+\langle u,\varepsilon\rangle$ and:
\begin{itemize}
	\item $\langle u,\varepsilon\rangle$ is a centered Gaussian random variable with variance $\sigma^2$,
	\item $h_G(u)$ is the endpoint of the distribution of $\langle u,X\rangle$.
\end{itemize}

In the sequel, we denote by $F_u$ the cumulative distribution function of $\langle u,X\rangle$. 

Consider the following assumption, which entails the next lemma.

\begin{assumption} \label{A0} 
	$B(a,r)\subseteq G\subseteq B(0,R)$, for some $a\in\R^d$. 
\end{assumption}

\begin{lemma} \label{lemma0}
	Let $G$ satisfy Assumption \ref{A0}. Then, for all $u\in\Sd$, $\theta_{F_u}=h_G(u)$ and $F_u\in\mathcal F(d,1,r,L)$, where $\DS L=(2R)^{d-1}r^d\kappa_d\max\left(1,\frac{d}{r^{d-1}\kappa_{d-1}}\right)$.
\end{lemma}

Hence, projecting the data $Y_j$, $ 1\leq j \leq n \dots, n $ on any direction brings us back to the one dimensional setup studied in Section \ref{Section1D}, where the end point of the corresponding distribution is the value of the support function of $G$ in the projection direction.

We are now in a position to define an estimator of $G$. For $u\in\R^d$, let $\hat h(u)$ be the estimator of $h_G(u)$ defined as $\DS \hat h(u)=\max_{1\leq j \leq n}\langle u,Y_j\rangle -b_n$, where we recall that $b_n=\sqrt{2\sigma^2 \ln n}$.

Let $M$ be a positive integer and $U_1,\ldots,U_M$ be independent uniform random vectors on the sphere $\Sd$ and define 
\begin{equation}
	\hat G_M=\{x\in\R^d:\langle U_j,x\rangle\leq \hat h(U_j), \; \forall j=1,\ldots,M\}.
\end{equation}
We also define a truncated version of $\hat G_M$. Let $\hat \mu_n=\frac{1}{n}\sum_{j=1}^n Y_j$. Define 
\begin{equation} \label{eq:truncated_estimator}
	\tilde G_M=\begin{cases} \hat G_M\cap B(\hat \mu_n,\ln n) \mbox{ if } \hat G_M\neq\emptyset \\ \{\hat \mu_n\} \mbox{ otherwise}. \end{cases}
\end{equation}

First, we give a deviation inequality for the estimator $\hat G_M$. As a corollary, some choice of $M$ (independent of $G$) will make the risk of the truncated estimator $\tilde G_M$ have order $(\ln n)^{-1/2}$.

\begin{theorem} \label{Theorem2}

Let $n>3$, $b_n=\sqrt{2\sigma^2\ln n}$ and $M$ be a positive integer with $(\ln M)/b_n\leq \min(r/(4\sigma^2),1/2)$. Then, there exist positive constants $c_0, c_1, c_2$ and $c_3$ such that the following holds. For all convex bodies $G$ that satisfy Assumption \ref{A0}, for all positive $x$ with $x\leq \frac{rb_n}{4\sigma^2}-\ln M$, 
$$\dH(\hat G_M,G)\leq c_0\frac{x+\ln M}{b_n}$$
with probability at least $1-c_1e^{-x}-Me^{-c_2 n}-(6b_n)^de^{-c_3M(\ln M)^{d-1}b_n^{-(d-1)}}$.

\end{theorem}

\begin{proof}[Proof of Theorem \ref{Theorem2}]

The proof relies on Lemma 7 in \cite{Brunel2016-2}, which we state here in a simpler form.

\begin{lemma} \label{TDLS}
Let $\delta\in (0,1/2]$ and $\mathcal N$ be a $\delta$-net of $\Sd$. Let $G$ be a convex body in $\R^d$ and $h_G$ its support function. Let $a\in\R^d$ and $0<r\leq R$ such that $B(a,r)\subseteq G\subseteq B(a,R)$. Let $\hat h:\Sd\to\R$ and $\hat G=\{x\in\R^d:\langle u,x\rangle \leq \hat h(u), \; \forall u\in\mathcal N\}$. Let $t=\max_{u\in\mathcal N}|\hat h(u)-h_G(u)|$. If $t\leq r/2$, then $\DS d_{\textsf H}(\hat G,G)\leq \frac{3t R}{2r}+4R\delta$.
\end{lemma}

Let $G$ satisfy Assumption \ref{A0}. Combining Lemma \ref{lemma0} and Theorem \ref{Theorem1}, we have that for all $u\in\mathbb S^{d-1}$, and all $t\geq 0$,
\begin{align} \label{Proof2_1}
	\PP_G\left[|\hat h(u)-h_G(u)|>t\right]\leq c_1e^{-\frac{b_n t}{2\sigma^2}}+e^{-c_2 n},
\end{align}
with $c_1$ and $c_2$ as in Theorem \ref{Theorem1} with $\alpha=(d+1)/2$. Hence, by a union bound, 
\begin{align} \label{Proof2_2}
	\PP_G\left[\max_{1\leq j \leq M}|\hat h(U_j)-h_G(U_j)|>t\right]\leq c_1Me^{-\frac{b_n t}{2\sigma^2}}+Me^{-c_2 n}.
\end{align}

Let $t<r/2$. Consider the event $\mathcal A$ where $U_1,\ldots,U_M$ form a $\delta$-net of $\Sd$, where $\delta\in (0,1/2)$. By Lemma \ref{TDLS}, if $\mathcal A$ holds and if $|\hat h(U_j)-h_G(U_j)|\leq t$ for all $j=1,\ldots,M$, then $\dH(\hat G,G)\leq \frac{3t R}{r}+4R\delta$. Hence, by \eqref{Proof2_2} and Lemma 10 in \cite{Brunel2016-2},
\begin{align}
	& \PP\left[\dH(\hat G,G)>\frac{3t R}{r}+4R\delta\right] \nonumber \\
	\label{EndProofThm} & \quad\quad \leq c_1Me^{-\frac{b_n t}{2\sigma^2}}+Me^{-c_2 n}+6^d\exp\left(-c_3 M\delta^{d-1}+d\ln\left(\frac{1}{\delta}\right)\right),
\end{align}
where $c_3=(2d8^{(d-1)/2})^{-1}$. Taking $\delta=(\ln M)/b_n$ ends the proof of Theorem \ref{Theorem2}.

\end{proof}

Theorem \ref{Theorem2} yields a uniform upper bound on the risk of $\tilde G_M$, which we derive for a special choice of $M$. Denote by $\mathcal K_{r,R}$ the collection of all convex bodies satisfying Assumption \ref{A0}.

\begin{corollary} \label{MainCor}

Let $A=2d(d+1)8^{(d-1)/2}$ and $M=\lfloor Ab_n^{d-1}(\ln b_n)^{-(d-2)}\rfloor$. Then, the truncated estimator $\tilde G_M$ satisfies
$$\sup_{G\in\mathcal K_{r,R}}\E_G[\dH(\tilde G_M,G)]=O\left(\frac{\ln\ln n}{\sqrt{\ln n}}\right).$$

\end{corollary}

\begin{remark}
Suppose that for all $x\in \partial G$, there exist $a,b\in\R^d$ such that $B(a,r)\subseteq G\subseteq B(b,R)$, $x\in B(a,r)$ and $x\in\partial B(b,R)$. In particular, this means that the complement of $G$ has \emph{reach} at least $r$, i.e., one can roll a Euclidean ball of radius $r$ inside $G$ along its boundary (see, e.g., \cite[Definition 11]{Thale2008}). In addition, $G$ can roll freely inside a Euclidean ball of radius $R$, along its boundary. This ensures that for all $u\in\mathbb S^{d-1}$, the random variable $\langle u,X\rangle-h_G(u)$ satisfies the assumption of Theorem \ref{Theorem1BIS} with $\alpha=(d+1)/2$ and some $L>0$ that depends on $r$ and $R$ only. Hence, we are in the case where $\alpha=\beta$ in Theorem \ref{Theorem1BIS}, which shows that the rate of estimation of the support function of $G$ at a single unit vector can be improved by a sublogarithmic factor. However, a close look at the proof of Theorem \ref{Theorem2} suggests that a sublogarithmic factor is still unavoidable in our proof technique, because of the union bound on a covering of the unit sphere.

\end{remark}

\begin{remark}
	Theorem \ref{Theorem2} can be easily extended to cases where the $X_j$'s are not uniformly distributed on $G$. What matters to the proof is that uniformly over unit vectors $u$, the cumulative distribution function $F_u$ of $\langle u,X\rangle-h_G(u)$ increases polynomially near zero. Examples of such distributions are given in \cite{BrunelHausdorff2017}.
\end{remark}

\begin{remark}
	Note that in general, the estimate $\hat h$ defined above is not a support function. In particular, it is not enough to control the differences $\hat h(U_j)-h_G(U_j)$, $ j=1,\ldots,M $ in order to obtain a bound on the Hausdorff distance between $\hat G_M$ and $G$.
\end{remark}

The next theorem gives a lower bound for the minimax risk of estimation $G\in\mathcal K_{r,R}$ that is also polylogarithmic in the sample size.

\begin{theorem}\label{thm:lower}
Let $ r $ and $ R $ be any two positive real numbers satisfying $ R/r \geq 2\sqrt{d} $. For each $ \tau $ in $ (0,1) $, there exist positive constants $ c $ and $ C $ depending only on $ d $, $ \sigma $, $ \tau $, $ r $, and $ R $ such that
\begin{equation*} \inf_{\hat{G}_n}\sup_{G\in\mathcal K_{r,R}}\mathbb{P}_G[\dH(G, \hat{G}_n) > c(\ln n)^{-2/\tau}] \geq C,
\end{equation*} 
and
\begin{equation*} \inf_{\hat{G}_n}\sup_{G\in\mathcal K_{r,R}}\mathbb{E}_G[\dH(G, \hat{G}_n)] \geq C(\ln n)^{-2/\tau},
\end{equation*} 
where the infimum runs over all estimators $ \hat{G}_n $ of $ G $ based on $ Y_1,\dots,Y_n $.
\end{theorem}

\begin{proof}[Proof of Theorem \ref{thm:lower}]
In the following, we assume that $ c $ and $ C $ are generic positive constants, depending only on $ d $, $ \sigma $, $ \tau $, and $ \delta $.

Let $ \delta > 0 $ be fixed and $ m $ be a positive integer. Let $ \psi $ be chosen as in Lemma \ref{psi_exist} and $ \gamma_m = (4/3)\delta^{-1} \pi m  $. Replacing $ \psi $ by $ x \mapsto 2\delta\psi(x/(2\delta)) $, we can assume that $ \psi $ is supported in the interval $ [-\delta, \delta] $ and $ \inf_{|x| \leq \delta(3/4)}\psi(x) > 0 $. Note that this transformation does not affect the bound on its derivatives \eqref{eq:psi-derivatives} and hence the decay of its Fourier transform remains unchanged.

Define $ h_m(x) = \psi(x)\sin(\gamma_mx) $, $ H_m(x_1,\dots,x_{d-1}) = \prod_{k=1}^{d-1}h_m(x_k) $, and for $ L > 0 $ and $ \omega \in \{-1, +1 \} $, let
\begin{equation*}
b_{\omega}(x_1,\dots,x_{d-1}) = \sum_{k=1}^{d-1} g(x_k) + \omega \;(L/\gamma_m^{2})H_m(x_1,\dots,x_{d-1}),
\end{equation*}
where $ g $ satisfies:
\begin{align}
& \max_{x\in[-\delta, \delta]}g(x) < \frac{\delta}{2(d-1)}, \quad \text{and} \label{eq:prop1} \\
& \max_{x\in[-\delta, \delta]}g^{\prime\prime}(x) < 0, \quad \text{and} \label{eq:prop0} \\
& |\mathcal{F}[g](t)| \leq Ce^{-c|t|^{\tau}}, \quad \text{for some positive constants}\; c \;\text{and}\; C \label{eq:prop2} 
\end{align}
For concreteness, one can take an appropriately scaled Cauchy density, $ g(x) \propto \frac{1}{1+x^2/\delta_0^2} $, which is strictly concave in the region where $ |x| < \delta_0/\sqrt{3} $ and satisfies \eqref{eq:prop1} with $ \delta_0 > \sqrt{3}\delta $ and \eqref{eq:prop2} with $ \tau = 1 $. From the inequality $1+|t|\geq |t|^\tau$,  we have that~\eqref{eq:prop2} is satisfied for all $\tau\in (0,1)$.

By \eqref{eq:prop1} and Lemma \ref{lmm:smooth}, we ensure that the Hessian of $ b_{\omega} $, i.e., $ \nabla ^2 b_{\omega} $, is negative-semidefinite and so that the sets
\begin{equation*} G_{\omega} = \{ (x_1, \dots, x_d)^{\prime} \in [-\delta,\delta]^{d-1}\times \R  : -\delta \leq x_d \leq b_{\omega}(x_1,\dots, x_{d-1})  \}
\end{equation*}
are convex. By choosing $ L < \frac{\gamma^2_m}{2\|\psi\|_{\infty}^{d-1}} $, we have $ (L/\gamma^2_m)|H_m(x_1,\dots,x_{d-1})| \leq L\|\psi\|_{\infty}^{d-1}/\gamma^2_m < \delta/2 $. Combining this with \eqref{eq:prop0}, we have $ |b_{\omega}| \leq \delta $. This means that $ G_{\omega} \subset [-\delta, \delta]^d $, and since $ [-\delta, \delta]^d \subset B_d(0, \sqrt{d}\delta) $, we may take $ R = \sqrt{d}\delta $. Finally, observe that $ B_d(-\delta/2, \delta/2) \subset G_{\omega} $, since the cube $ [-\delta, 0]^d $ is contained in $ G_{\omega} $. Thus, we may take $ r = \delta/2 $. With these choices of $ r $ and $ R $, we have $ G_{\omega} \in \mathcal K_{r,R} $.


Note that $ h_m $ is an odd function about the origin. Thus $ \int_{[-\delta,\delta]^{d-1}}H_m(x)dx= 0 $ because we are integrating an odd function about the origin. Therefore, $ |G_{\omega}| = \delta(2\delta)^{d-1} + (d-1)\int_{[-\delta,\delta]}g(x)dx $.
Also, note that
\begin{align*}  
\mathsf{d}_{\Delta}(G_{+1}, G_{-1})
& = \int_{[-\delta,\delta]^{d-1}}|b_{+1}(x) - b_{-1}(x)|dx \\
& = \dfrac{2L}{\gamma^{2}_m}\int_{[-\delta,\delta]^{d-1}}|H_m(x)|dx \\
& = \dfrac{2L}{\gamma^{2}_m}\prod_{k=1}^{d-1}\int_{[-\delta,\delta]}|\sin(\gamma_mx_k)\psi(x_k)|dx_k.
\end{align*}
The factor $ \prod_{k=1}^{d-1}\int_{[-\delta,\delta]}|\sin(\gamma_mx_k)\psi(x_k)|dx_k $ in the above expression can be lower bounded by a constant, independent of $ m $. In fact,
\begin{align*}
\int_{[-\delta,\delta]}|\sin(\gamma_mx_k)\psi(x_k)|dx_k 
& \geq \int_{|x_k|\leq \delta(3/4)}|\sin(\gamma_mx_k)\psi(x_k)|dx_k \\
& \geq 3\delta/4 \inf_{|x|\leq \delta(3/4)}|\psi(x)|\int_{|x_k|\leq 1}|\sin(\pi m x_k)|dx_k \\
& = 3\delta/\pi \inf_{|x|\leq \delta(3/4)}|\psi(x)|  \\
& > 0.
\end{align*}
Here, we used the fact that
\begin{align*} 
\int_{[-1,1]}|\sin(\pi mx)|dx 
& = 4m\int_{[0,1/(2m)]}|\sin(\pi mx)|dx \\
& = (4/\pi)\int_{[0,\pi/2]}\sin(x)dx \\
& = 4/\pi,
\end{align*}
for any non-zero integer $ m $.
Thus, there exists a constant $ C_1 > 0 $, independent of $ m $, such that
\begin{equation}
\mathsf{d}_{\Delta}(G_{+1}, G_{-1}) \geq \dfrac{C_1}{m^2}.
\end{equation}

For $ \omega = \pm 1 $, define $ f_{\omega} = \mathbbm{1}_{G_{\omega}}/|G_{\omega}| $. Note that for all $ y > 0 $,
\begin{align*}
\text{TV}(\mathbb{P}_{G_{+1}}, \mathbb{P}_{G_{-1}}) & = \frac{1}{2}\int_{\mathbb{R}^d}|(f_{+1}-f_{-1})\ast \phi_{\sigma}(x)|dx \\ & = \frac{1}{2}\int_{\|x\|>y}|(f_{+1}-f_{-1})\ast \phi_{\sigma}(x)|dx + \frac{1}{2}\int_{\|x\|\leq y}|(f_{+1}-f_{-1})\ast \phi_{\sigma}(x)|dx
 \\ & \leq \int_{\|x\|>y}\sup_{z\in[-\delta, \delta]^d}\phi_{\sigma}(x-z)dx + \\ & \qquad \frac{1}{2}\sqrt{|B_d(0,y)|}\sqrt{\int_{\mathbb{R}^d}|\mathcal{F}[f_{+}-f_{-1}](t)\mathcal{F}[\phi_{\sigma}](t)|^2dt}  \\ & \leq C_2e^{-c_2y^2}+C_2y^{d/2}\sqrt{\int_{\mathbb{R}^d}|\mathcal{F}[f_{+}-f_{-1}](t)\mathcal{F}[\phi_{\sigma}](t)|^2dt},
\end{align*}
for some positive constants $ c_2 $ and $ C_2 $ that depend only on $ \delta $, $ \sigma $, and $ d $. Set $ y \propto \sqrt{\log \frac{1}{\int_{\mathbb{R}^d}|\mathcal{F}[f_{+}-f_{-1}](t)\mathcal{F}[\phi_{\sigma}](t)|^2dt}} $ so that $ \text{TV}(\mathbb{P}_{G_{+1}}, \mathbb{P}_{G_{-1}}) $ can be bounded by a fixed power of $ \int_{\mathbb{R}^d}|\mathcal{F}[f_{+}-f_{-1}](t)\mathcal{F}[\phi_{\sigma}](t)|^2dt $.

Split $ \int_{\mathbb{R}^d}|\mathcal{F}[f_{+1}-f_{-1}](t)\mathcal{F}[\phi_{\sigma}](t)|^2dt $ into two integrals with domains of integration $ \|t\|_{\infty} \leq am^{\tau} $ and $ \|t\|_{\infty} > am^{\tau} $. Using the fact that $ \mathcal{F}[\phi_{\sigma}](t) = \sigma^de^{-\sigma^2\|t\|^2_2/2} $, we have
\begin{align*}
\int_{\|t\|_{\infty} > am^{\tau} }|\mathcal{F}[f_{+1}-f_{-1}](t)\mathcal{F}[\phi_{\sigma}](t)|^2dt 
& \leq C_3e^{-c_3m^{2\tau}}.
\end{align*}
By Lemma \ref{bound1}, we have 
\begin{equation*}
|\mathcal{F}[f_{+1}-f_{-1}](t)| \leq Ce^{-cm^{\tau}},
\end{equation*}
whenever $ \|t\|_{\infty} \leq am^{\tau} $.
Thus
\begin{align*}
\int_{\|t\|_{\infty}\leq am^{\tau}}|\mathcal{F}[f_{+1}-f_{-1}](t)\mathcal{F}[\phi_{\sigma}](t)|^2dt & \\
\leq Ce^{-cm^{\tau}}\int_{\mathbb{R}^d}|\mathcal{F}[\phi_{\sigma}](t)|^2dt.
\end{align*}
This shows that 
\begin{equation*}
\text{TV}(\mathbb{P}_{G_{+1}}, \mathbb{P}_{G_{-1}}) \leq C_4e^{-c_4m^{\tau}},
\end{equation*}
for some positive constants $ c_4 $ and $ C_4 $ that depend only on $ d $, $ \sigma $, $ \tau $, and $ \delta $.

The lower bound is a simple two point statistical hypothesis test. By Lemma \ref{measures},
\begin{align*} 
& \inf_{\hat{G}_n}\sup_{G\in\mathcal K_{r,R}}\mathbb{P}_G[C_5\dH(G, \hat{G}_n) > c_5(\ln n)^{-2/\tau}] \geq \\ &\qquad \inf_{\hat{G}_n}\sup_{G\in\mathcal K_{r,R}}\mathbb{P}_G[\mathsf{d}_{\Delta}(G, \hat{G}_n) > c_6(\ln n)^{-2/\tau}].
\end{align*} 
In summary, we have shown that $ \mathsf{d}_{\Delta}(G_{+1}, G_{-1}) \geq \tfrac{C_1}{m^{2}} $ and $ \text{TV}(\mathbb{P}_{G_{+1}}, \mathbb{P}_{G_{-1}}) \leq C_4e^{-c_4m^{\tau}} $, where the constants depend only on $ d $, $ \sigma $, $ \tau $, $ \delta $. Choosing $ m \asymp (\ln n)^{1/\tau} $ and applying Theorem 2.2(i) in \cite{Tsybakov2008} finishes the proof of the lower bound on the minimax probability. To get the second conclusion of the theorem, apply Markov's inequality.
\end{proof}

\section{Discussion}

\subsection{Gap between the lower and upper bounds}


Note that our upper (Corollary \ref{MainCor}) and lower (Theorem \ref{thm:lower}) bounds do not match, in the same way as in \cite[Section 5]{Wasserman2012}. Like these authors, we do not know how to close the gap at the moment. However, both our bounds are very slow: They decay at a polylogarithmic rate. This is not surprising as rates are very slow in general for ill-posed deconvolution problems. However, perhaps more surprisingly, the rates are dimension independent, only the multiplicative constants are exponentially large in $d$ in the upper bound.

\subsection{Other noise distributions}

The techniques that we use here can easily be extended to other noise distributions, provided they are known. Let us look, for instance, at the case when the noise terms are bounded, e.g., uniform in some ball.

Let $G$ satisfy Assumption \ref{A0} and suppose that $ \varepsilon $ is uniformly distributed on the ball $ B(0, Q) $, where $ Q > r $ is known. We can use
\begin{equation*}
\hat{h}_n(u) = \max_{1\leq j \leq n} \langle u, Y_j \rangle - Q
\end{equation*}
to form the truncated estimator $ \tilde{G}_M $ from \eqref{eq:truncated_estimator}.

First note that for all unit vectors $u$, the density $f_{\langle u, \varepsilon \rangle}$ of $\langle u,\varepsilon\rangle$ satisfies 
\begin{equation}  \label{lemmaball}
f_{\langle u, \varepsilon \rangle}(x) \geq C(Q-|x|)^{\frac{d-1}{2}},
\end{equation} 
for all real number $x$ with $ |x| \leq Q $, where $ C $ is a positive constant that does not depend on $u$. This yields the following lemma.

\begin{lemma} \label{lemmaDiscussion}
Let $G$ satisfy Assumption \ref{A0}. For $u\in \mathbb S^{d-1}$ and $x\in\R$, let $ G_u(x) = \mathbb{P}[\langle u, Y \rangle - h_G(u) \leq x] $. There exists a positive constant $ c $ that depends only on $ r $, $ R $, and $ d $ such that for all real numbers $x$ with $ Q-r \leq x \leq Q $,
\begin{equation*}
1 - G_u(x) \geq c (Q-x)^{\frac{3d+1}{2}}.
\end{equation*} 
\end{lemma}

\begin{proof}[Proof of Lemma \ref{lemmaDiscussion}]
Suppose $Q-r\leq x\leq Q$ and let $L$ be as in Lemma \ref{lemma0}. Then,
\begin{align*}
1-G_u(x) & = \int_{-\infty}^{0} (1-F_u(t)) f_{\langle u, \varepsilon \rangle}(x-t)dt \\
& \geq \int_{-r}^{0} (1-F_u(t)) f_{\langle u, \varepsilon \rangle}(x-t)dt \\
& \geq CL^{-1}\int_{x-Q}^{0}(-t)^{d}(Q-x+t)^{\frac{d-1}{2}}dt \\
& = CL^{-1}\int_{0}^{Q-x}t^{d}(Q-x-t)^{\frac{d-1}{2}}dt \\
& = CL^{-1}\mathcal B\left(d+1, \frac{d+1}{2}\right)(Q-x)^{\frac{3d+1}{2}},
\end{align*}
where $\mathcal B$ is the Beta function.
\end{proof}

Next, observe that for all $t\geq 0$,
\begin{align*}
\mathbb{P}[ |\hat{h}_n(u)-h_G(u)| > t ] 
& = 1-G_u(Q+t)^n + G_u(Q-t)^n \\
& = G_u(h_G(u)+Q-t)^n \\
& \leq \left(1-ct^{\frac{3d+1}{2}}\right)^n \\
& \leq e^{-cnt^{\frac{3d+1}{2}}},
\end{align*} 
where $\DS c=CL^{-1}\mathcal B\left(d+1,\frac{d+1}{2}\right)$, see Lemma \ref{lemmaDiscussion}. Hence, an adaptation of the proof of Theorem \ref{TheoremMinimax} yields the following.

\begin{theorem}
Let the noise terms $\varepsilon_1,\ldots,\varepsilon_n$ be i.i.d. uniformly distributed in the ball $B(0,Q)$, for some known $Q>0$. Then,

\begin{equation*}
\sup_{G\in\mathcal{K}_{r, R}}\mathbb{E}_G [d_H(G, \tilde{G}_M)] = O((\ln\ln n)n^{-\frac{2}{3d+1}}).
\end{equation*} 

\end{theorem}

\section{Appendix} \label{SectionProofs}

\subsection{Proof of Corollary \ref{MainCor}}

In the sequel, let $a\in B_d(0,R)$ coming from Assumption \ref{A0}. Note that since $\tilde G_M\subseteq B(\hat \mu_n,\ln n)$ and $G\subseteq B(0,R)$, 
\begin{equation} \label{upperboundd_H}
	\dH(\tilde G_M,G)\leq |\hat\mu_n-a|+\ln n+R\leq |\hat\mu_n-\mu|+\ln n+2R,
\end{equation}
where $\mu$ is the centroid of $G$. Consider the events $\mathcal A$: ``$\hat G_M\neq\emptyset$" and $\mathcal B$: ``$|\hat\mu_n-\mu|\leq 5R$". Write
\begin{equation} \label{ProofMainCor0}
	\E_G[\dH(\tilde G_M,G)]=E_1+E_2+E_3,
\end{equation}
where $E_1=\E_G[\dH(\tilde G_M,G)\mathds 1_{\mathcal A\cap\mathcal B}]$, $E_2=\E_G[\dH(\tilde G_M,G)\mathds 1_{\mathcal A^{\complement}\cap\mathcal B}]$ and $E_3=\E_G[\dH(\tilde G_M,G)\mathds 1_{\mathcal B^{\complement}}]$.
In order to bound $E_1$, let us state the following lemma, which is a simple application of Fubini's lemma.
\begin{lemma}
	Let $Z$ be a nonnegative random variable and $A$ a positive number. Then,
	$$\E[Z\mathds 1_{Z<A}]\leq \int_0^A\PP[Z\geq t]\diff t.$$
\end{lemma}
This lemma yields, together with \eqref{upperboundd_H}, with the same notation as in \eqref{EndProofThm},
\begin{align}
	E_1 & \leq \int_0^{\ln n+7R}\PP[\dH(\tilde G_M,G)\geq t] \nonumber \\
	& \leq 4R\delta+\int_0^{\ln n+7R-4R\delta}\PP[\dH(\tilde G_M,G)\geq t+4R\delta] \nonumber \\
	& = 4R\delta+\frac{3R}{r}\int_0^{r(\ln n)/(3R)+7r/3-4r\delta/3}\PP[\dH(\tilde G_M,G)\geq \frac{3Rt}{r}+4R\delta]. \label{ProofMainCor1}
\end{align}
Now, we split the last integral in \eqref{ProofMainCor1} in two terms: First, the integral between $0$ and $r/2$, where we can apply \eqref{EndProofThm}, and then between $r/2$ and $r(\ln n)/(3R)+7r/3-4r\delta/3$, where we bound the probability term by the value it takes for $t=r/2$. This yields
\begin{equation} \label{ProofE_1}
	E_1\leq \frac{C_1\ln \ln n}{\sqrt{\ln n}},
\end{equation}
for some positive constant $C_1$ that depends neither on $n$ nor on $G$. For $E_2$, note that if $\mathcal A$ is not satisfied, then $\tilde G_M=\{\hat\mu_n\}$ and $\dH(\tilde G_M,G)\leq |\hat\mu_n-\mu|+2R$, which is bounded from above by $7R$ is $\mathcal B$ is satisfied. Hence,
\begin{align}
	E_2 & \leq 7R\PP[\hat G_M=\emptyset] \nonumber \\
	& \leq 7R\PP[a\notin\hat G_M] \nonumber \\
	& = 7R\PP[\exists j=1,\ldots,M: \hat h(U_j)<\langle U_j,a\rangle] \nonumber \\
	& \leq 7RM\PP[\hat h(U_1)<\langle U_1,a\rangle] \nonumber \\
	& \leq 7RM\PP[\hat h(U_1)<h_G(U_1)-r/2] \nonumber \\
	& \leq 7RMc_1e^{-\frac{b_n r/2}{2\sigma^2}}+e^{-c_2 n} \nonumber
\end{align}
by \eqref{Proof2_1}.
Hence, 
\begin{equation} \label{ProofE_2}
	E_2\leq \frac{C_2\ln\ln n}{\sqrt{\ln n}},
\end{equation}
where $C_2$ is a positive constant that depends neither on $n$ nor on $G$. Now, using \eqref{upperboundd_H},
\begin{equation} \label{ProofMainCor3}
	E_3\leq \E_G\left[\left(|\hat\mu_n-\mu|+\ln n+2R\right)\mathds 1_{|\hat\mu_n-\mu|>R}\right].
\end{equation}
To bound the latter expectation from above, we use the following lemma, which is also an direct application of Fubini's lemma.
\begin{lemma}
	Let $Z$ be a nonnegative random variable and $A$ a positive number. Then,
	$$\E[Z\mathds 1_{Z>A}]\leq A+\int_A^\infty \PP[Z\geq t]\diff t.$$
\end{lemma}
Hence, \eqref{ProofMainCor3} yields
\begin{equation} \label{ProofMainCor4}
	E_3\leq (\ln n+3R)\PP[|\hat\mu_n-\mu|>5R]+\int_{5R}^\infty \PP[|\hat\mu_n-\mu|\geq t]\diff t.
\end{equation}

We now use the following lemma.

\begin{lemma} \label{lemmaConcentra}
	For all $t\geq 5R$,
	$$\PP[|\hat\mu_n-\mu|>t]\leq 6^d e^{-9nt^2/200}.$$
\end{lemma}

\begin{proof}[Proof of Lemma \ref{lemmaConcentra}]
	Let $\mathcal N$ be a $(1/2)$-net of the unit sphere. Let $u\in\Sd$ such that $|\hat \mu_n-\mu|=\langle u,\hat\mu_n-\mu\rangle$. Let $u^*\in\mathcal N$ such that $|u^*-u|\leq 1/2$. Then, by Cauchy-Schartz inequality,
\begin{align*}
	\langle u^*,\mu\rangle & \geq \langle u,\hat\mu_n-\mu\rangle - (1/2)|\hat\mu_n-\mu| \\
	& = \frac{1}{2}|\hat\mu_n-\mu|.
\end{align*}
Hence,
\begin{align}
	\PP[|\hat\mu_n-\mu|>t] & \leq \PP[\exists u^*\in\mathcal N: \langle u^*,\hat\mu_n-\mu\rangle \geq t/2] \nonumber \\
	& \leq 6^d \max_{u\in\mathcal N}\PP[\langle u,\hat\mu_n-\mu\rangle \geq t/2] \nonumber \\
	& \leq 6^d \max_{u\in\Sd}\PP[\langle u,\hat\mu_n-\mu\rangle \geq t/2].	\label{ProofMainCor40}
\end{align}
Let $u\in\Sd$. Then, by Markov's inequality, and using the fact that $|X_1-\mu|\leq 2R$ almost surely, for all $\lambda>0$, 
\begin{align*}
	\PP[\langle u,\hat \mu_n-\mu\rangle \geq t/2] & \leq \E\left[e^{\frac{\lambda\langle u,Y_1-\mu\rangle}{n}}\right]^n e^{-\lambda t/2} \\
	& \leq \E\left[e^{\frac{\lambda\langle u,X_1-\mu\rangle}{n}}\right]^n \E\left[e^{\frac{\lambda\langle u,\varepsilon_1\rangle}{n}}\right]^n e^{-\lambda t/2} \\ 
	& \leq e^{2R\lambda+\lambda^2\sigma^2/(2n)}e^{-\lambda t/2}.	
\end{align*}
Choosing $\lambda=\frac{3nt}{10\sigma^2}$ and plugging in \eqref{ProofMainCor40} yields the desired result.
\end{proof}

Applying Lemma \ref{lemmaConcentra} to \eqref{ProofMainCor4} entails 
\begin{equation} \label{ProofE_3}
	E_3\leq \frac{C_3\ln\ln n}{\sqrt{\ln n}}.
\end{equation}
Applying \eqref{ProofE_1}, \eqref{ProofE_2} and \eqref{ProofE_3} to \eqref{ProofMainCor0} ends the proof of the corollary. \hfill \textsquare

\subsection{Intermediate lemmas and their proofs}\label{SectionProofLemmas}

\paragraph{Proof of Lemma \ref{Lemmacdf}:}

Without loss of generality, let us assume that $\theta_F=0$. For all $x\in\R$,
\begin{equation} \label{Proof1_1}
	1-G(x) = \int_{-\infty}^0 \left(1-F(t)\right)\frac{e^{\frac{(x-t)^2}{2\sigma^2}}}{\sqrt{2\pi\sigma^2}}\diff t.
\end{equation}
Let us split the latter integral into two parts: Denote by $I_1$ the integral between $-\infty$ and $-r$ and by $I_2$ the integral between $-r$ and $0$, so $1-G(x)=I_1+I_2$.

Let $x> 0$. First, using the assumption about $F$, one has:
\begin{align*}
	I_1 & = \int_0^r \left(1-F(-t)\right)\frac{e^{-\frac{(x+t)^2}{2\sigma^2}}}{\sqrt{2\pi\sigma^2}}\diff t \nonumber \\
	& \leq \frac{L}{\sqrt{2\pi\sigma^2}}\int_0^r t^\alpha e^{-\frac{(x+t)^2}{2\sigma^2}} \diff t \nonumber \\
	& = \frac{Le^{-\frac{x^2}{2\sigma^2}}}{\sqrt{2\pi\sigma^2}}\int_0^r t^\alpha e^{\frac{-xt}{\sigma^2}}e^{\frac{-t^2}{2\sigma^2}} \diff t \nonumber \\
	& \leq \frac{L\sigma^{2\alpha+2}e^{-\frac{x^2}{2\sigma^2}}}{x^{\alpha+1}\sqrt{2\pi\sigma^2}}\int_0^{rx/\sigma^2} t^\alpha e^{-t} \diff t \nonumber \\
	& \leq \frac{L\Gamma(\alpha+1)\sigma^{2\alpha+1}e^{-\frac{x^2}{2\sigma^2}}}{x^{\alpha+1}\sqrt{2\pi}},
\end{align*}
where $\Gamma$ is Euler's gamma function. Hence, 
\begin{equation} \label{Proof1_2}
	I_1\leq \frac{C'e^{-\frac{x^2}{2\sigma^2}}}{x^{\alpha+1}},
\end{equation}
where $\DS C'= \frac{L\Gamma(\alpha+1)\sigma^{2\alpha+1}}{\sqrt{2\pi}}$ is a positive constant.
On the other hand, if $x\geq \sigma^2/r$,
\begin{align*}
	I_1 & = \int_0^r \left(1-F(-t)\right)\frac{e^{-\frac{(x+t)^2}{2\sigma^2}}}{\sqrt{2\pi\sigma^2}}\diff t \nonumber \\
	& \geq \frac{L^{-1}}{\sqrt{2\pi\sigma^2}}\int_0^r t^\alpha e^{-\frac{(x+t)^2}{2\sigma^2}} \diff t \nonumber \\
	& = \frac{L^{-1}e^{-\frac{x^2}{2\sigma^2}}}{\sqrt{2\pi\sigma^2}}\int_0^r t^\alpha e^{\frac{-xt}{\sigma^2}}e^{\frac{-t^2}{2\sigma^2}} \diff t \nonumber \\
	& \geq \frac{L^{-1}\sigma^{2\alpha+2}e^{-\frac{r^2}{2\sigma^2}}e^{-\frac{x^2}{2\sigma^2}}}{x^{\alpha+1}\sqrt{2\pi\sigma^2}}\int_0^{rx/\sigma^2} t^\alpha e^{-t} \diff t \nonumber \\
	& \geq \frac{L^{-1}e^{-\frac{r^2}{2\sigma^2}}\sigma^{2\alpha+2}e^{-\frac{x^2}{2\sigma^2}}}{x^{\alpha+1}\sqrt{2\pi\sigma^2}}\int_0^1 t^\alpha e^{-t} \diff t.
\end{align*}
Hence, 
\begin{equation} \label{Proof1_3}
	I_1\geq \frac{ce^{-\frac{x^2}{2\sigma^2}}}{x^{\alpha+1}},
\end{equation}
where $\DS c = \frac{L^{-1}e^{-\frac{r^2}{2\sigma^2}}\sigma^{2\alpha+2}}{\sqrt{2\pi\sigma^2}}\int_0^1 t^\alpha e^{-t} \diff t$ is a positive constant.

Now, we bound the nonnegative integral $I_2$ from above. Using the fact that $1-F(u)\leq u$ for all $u\in\R$,
\begin{align*}
	I_2 & = \int_r^{\infty} \left(1-F(-t)\right)\frac{e^{-\frac{(x+t)^2}{2\sigma^2}}}{\sqrt{2\pi\sigma^2}}\diff t \\
	& \leq \int_r^{\infty} \frac{e^{-\frac{(x+t)^2}{2\sigma^2}}}{\sqrt{2\pi\sigma^2}}\diff t \\
	& = \frac{e^{-\frac{x^2}{2\sigma^2}}}{\sqrt{2\pi\sigma^2}}\int_r^{\infty}e^{-\frac{xt}{\sigma^2}}e^{-\frac{t^2}{2\sigma^2}}\diff t \\
	& \leq e^{-\frac{x^2}{2\sigma^2}}e^{-\frac{xr}{\sigma^2}}\int_r^{\infty}\frac{e^{-\frac{t^2}{2\sigma^2}}}{\sqrt{2\pi\sigma^2}}\diff t \\
	& = \frac{1}{2}e^{-\frac{x^2}{2\sigma^2}}e^{-\frac{xr}{\sigma^2}}.
\end{align*}

Since for all $t\geq 0$, $\displaystyle{e^{-t}t^{\alpha+1}\leq\left(\frac{\alpha+1}{e}\right)^{\alpha+1}}$,
\begin{equation} \label{Proof1_4}
	I_2\leq \frac{C''e^{-\frac{x^2}{2\sigma^2}}}{x^{\alpha+1}},
\end{equation}
with $C''$ being the positive constant
\begin{equation*}
	C''=\frac{\sigma^{2\alpha+2}}{2r^{\alpha+1}}\left(\frac{\alpha+1}{e}\right)^{\alpha+1}.
\end{equation*}

Hence, \eqref{Proof1_2}, \eqref{Proof1_3} and \eqref{Proof1_4} yield
\begin{equation} \label{Proof1_5}
	\frac{ce^{-\frac{x^2}{2\sigma^2}}}{x^{\alpha+1}}\leq 1-G(x)\leq (C'+C'')\frac{e^{-\frac{x^2}{2\sigma^2}}}{x^{\alpha+1}},
\end{equation}
for all $x\geq \sigma^2/r$. This proves Lemma \ref{Lemmacdf}. \hfill \textsquare

\paragraph{Proof of Lemma \ref{lemma0}:}

Let $u\in\Sd$. For $t\geq 0$, denote by $C_G(u,t)=\{x\in G:\langle u,x\rangle\geq h_G(u)-t\}$. Then, for all $t\geq 0$, $\DS 1-F_u(t)=\frac{|C_G(u,t)|}{|G|}$. Let $x^*\in G$ such that $\langle u,x^*\rangle=h_G(u)$: $G$ has a supporting hyperplane passing through $x^*$ that is orthogonal to $u$.

By Assumption \ref{A0}, there is a ball $B=B(a,r)$ included in $G$. Consider the section $B_u$ of $B$ passing through $a$, orthogonal to $u$: $B_u=B\cap (a_u^{\perp})$. Denote by $\textsf{cone}$ the smallest cone with apex $x^*$ that contains $B_u$. Then, for all $t\in [0,r]$, $|C_G(u,t)|\geq |C_{\textsf{cone}}(u,t)|=\left(\frac{r}{\ell}\right)^{d-1}\frac{\kappa_{d-1}t^d}{d}$, where $\ell=\langle u,x^*-a\rangle$. Since $G\subseteq B(0,R)$ by Assumption \ref{A0}, $\ell\leq 2R$ and since $B(a,r)\subseteq G$, $|G|\geq r^d\kappa_d$, which altogether proves the lower bound of Lemma \ref{lemma0}. For the upper bound, note that Assumption \ref{A0} implies that $G$ can be included in a hypercube with edge length $2R$ that has one of its $(d-1)$-dimensional faces that contains $x^*$ and is orthogonal to $u$. Hence, $|C_G(u,t)|\leq 2Rt$, for all $t\in [0,2R]$. This proves the upper bound of Lemma \ref{lemma0}.

\begin{lemma} \label{measures}
If $ G $ and $ G^{\prime} $ are convex sets satisfying Assumption \ref{A0}, then there exists a constant $ C $ that depends only on $d$ and $R$ such that
\begin{equation*}
\mathsf{d}_{\Delta}(G, G^{\prime}) \leq C\dH(G, G^{\prime}).
\end{equation*}
\end{lemma}

\begin{proof}[Proof of Lemma \ref{measures}]
See Lemma 2 in \cite{Brunel2013}.
\end{proof}

\begin{lemma} \label{bound1}
Let $ G_{+1} $ and $ G_{-1} $ be the two convex sets from Theorem \ref{thm:lower}. There exists constants $ a > 0 $, $ c > 0 $ and $ C > 0 $, depending only on $ d $, $ \tau $, and $ \delta $, such that if $ \|t\|_{\infty} \leq am^{\tau} $, then
\begin{equation*}  |\mathcal{F}[\mathbbm{1}_{G_{+1}} - \mathbbm{1}_{G_{-1}}](t)| \leq Ce^{-cm^{\tau}}.
\end{equation*}
\end{lemma}

\begin{proof}[Proof of Lemma \ref{bound1}]
The ideas we use here are inspired by the proof of Theorem 8 in \cite{Wasserman2012}. Let $ t = (t_1,\dots,t_{d})^{\prime} $ belong to the product set
\begin{equation*}
[-\gamma_m/2, \gamma_m/2]^{d-1} \times [-am^{\tau}, am^{\tau}].
\end{equation*}
Note that
\begin{align}
& \mathcal{F}[\mathbbm{1}_{G_{+1}} - \mathbbm{1}_{G_{-1}}](t) \nonumber \\
& = \int_{[-\delta,\delta]^{d-1}}e^{i(t_1x_1+\cdots+t_{d-1}x_{d-1})}\dfrac{ e^{ib_{+1}(x_1,\dots,x_{d-1})t_d} - e^{ib_{-1}(x_1,\dots,x_{d-1})t_d}}{it_d}dx \nonumber \\
& = 2\int_{[-\delta,\delta]^{d-1}}e^{i(t_1x_1+\cdots+t_{d-1}x_{d-1})}e^{it_d\sum_{k=1}^{d-1}g(x_k)}\dfrac{\sin((Lt_d/\gamma^2_m)H(x))}{t_d}dx \nonumber \\
& = 2\sum_{j=0}^{\infty}\dfrac{(Lt_d/\gamma^2_m)^{2j+1}(-1)^j}{t_d(2j+1)!}\prod_{k=1}^{d-1}\int_{\mathbb{R}}e^{it_kx_k}e^{it_dg(x_k)}h^{2j+1}(x_k)dx_k \nonumber \\
& = 2\sum_{j=0}^{\infty}\dfrac{(Lt_d/\gamma^2_m)^{2j+1}(-1)^j}{t_d(2j+1)!}\prod_{k=1}^{d-1}(\mathcal{F}[\sin^{2j+1}(\gamma_mx_k)e^{it_dg(x_k)}\psi^{2j+1}(x_k)])(t_k) \label{eq:Fourier_main}.
\end{align}
Next, write
\begin{align*}
\sin^{2j+1}(\gamma_mx_k)  & = \left(\dfrac{e^{ix_k\gamma_m}-e^{-ix_k\gamma_m}}{2i}\right)^{2j+1} \\
& =  \left(\dfrac{1}{2i}\right)^{2j+1}\sum_{s=0}^{2j+1}\tbinom{2j+1}{s}(-1)^{s}e^{-ix_kw_s},
\end{align*}
where $ w_s = \gamma_m(2s-2j-1) $.

Using this expression and linearity of the Fourier transform, we can write
\begin{align*}
& (\mathcal{F}[\sin^{2j+1}(\gamma_mx_k)e^{it_dg(x_k)}\psi^{2j+1}(x_k)])(t_k) \\ & = 
\left(\dfrac{1}{2i}\right)^{2j+1}\sum_{s=0}^{2j+1}\tbinom{2j+1}{s}(-1)^s(\mathcal{F}[e^{it_dg(x_k)-ix_kw_s}\psi^{2j+1}(x_k)])(t_k) \\ &
= \left(\dfrac{1}{2i}\right)^{2j+1}\sum_{s=0}^{2j+1}\tbinom{2j+1}{s}(-1)^s(\mathcal{F}[e^{it_dg(x_k)}\psi^{2j+1}(x_k)])(t_k-w_s),
\end{align*}
and hence by the triangle inequality,
\begin{align}
& |(\mathcal{F}[\sin^{2j+1}(\gamma_mx_k)e^{it_dg(x_k)}\psi^{2j+1}(x_k)])(t_k)| \nonumber \\ 
& \leq \left(\dfrac{1}{2}\right)^{2j+1}\sum_{s=0}^{2j+1}\tbinom{2j+1}{s}|\mathcal{F}[e^{it_dg(x_k)}\psi^{2j+1}(x_k)](t_k-w_s)| \label{eq:Fourier_inequality}.
\end{align}
The function $
x \mapsto e^{i t_dg(x)} $ can be expanded as
\begin{equation*}
\sum_{\ell=0}^{\infty}\tfrac{(it_dg(x))^\ell}{\ell!},
\end{equation*} 
and hence
\begin{equation} \label{eq:inequalitysum}
|\mathcal{F}[e^{it_dg(x_k)}\psi^{2j+1}(x_k)](t_k-w_s)| \leq \sum_{\ell=0}^{\infty}\tfrac{|t_d|^\ell}{\ell!}|\mathcal{F}[g^\ell(x_k)\psi^{2j+1}(x_k)](t_k-w_s)|.
\end{equation}
By \eqref{eq:prop2}, $ g $ is chosen so that its Fourier transform has the same decay as the Fourier transform of $ \psi $. We deduce from Lemma \ref{psi_decay} that there exists constants $ c > 0 $ and $ B > 0 $, indepenent of $ j $ and $ \ell $, such that
\begin{equation*} 
|\mathcal{F}[g^\ell(x_k)\psi^{2j+1}(x_k)](t_k-w_s)| \leq B^{\ell+2j+1}e^{-c|t_k-w_s|^{\tau}}.
\end{equation*}
Applying this inequality to each term in the sum in \eqref{eq:inequalitysum} and summing over $ \ell $, we find that
\begin{equation*} 
|\mathcal{F}[e^{it_dg(x_k)}\psi^{2j+1}(x_k)](t_k-w_s)| \leq B^{2j+1}e^{B|t_d|-c|t_k-w_s|^{\tau}}.
\end{equation*} 
Since we restricted the $ t_k $ ($ k = 1,\dots,d-1 $) to be in the interval $ [-\gamma_m/2, \gamma_m/2] $, it follows that $ |t_k - w_s| \geq \gamma_m/2 $. Hence if $ \|t\|_{\infty} \leq am^{\tau} $, then
\begin{equation*}
|\mathcal{F}[e^{it_dg(x_k)}\psi^{2j+1}(x_k)](t_k-w_s)| \leq B^{2j+1}e^{Bam^{\tau}-c\gamma_m^{\tau}/2}.
\end{equation*}
Set $ a = c\gamma^{\tau}_m/(4Bm^{\tau}) $, which is independent of $ m $. Thus there exists a positive constant $ c_1 $ such that
\begin{equation}  \label{eq:inequality2}
|\mathcal{F}[e^{it_dg(x_k)}\psi^{2j+1}(x_k)](t_k-w_s)| \leq B^{2j+1}e^{-c_1m^{\tau}}.
\end{equation}
Finally, we apply the inequality \eqref{eq:inequality2} to each term in the sum in \eqref{eq:Fourier_inequality} and use the identity $ \left(\tfrac{1}{2}\right)^{2j+1}\sum_{s=0}^{2j+1}\tbinom{2j+1}{s} = 1 $ which yields
\begin{align} \label{eq:inequality3}
|(\mathcal{F}[\sin^{2j+1}(\gamma_mx_k)e^{it_dg(x_k)}\psi^{2j+1}(x_k)])(t_k)|
& \leq B^{2j+1}e^{-c_1m^{\tau}}.
\end{align}
Returning to \eqref{eq:Fourier_main}, we can use \eqref{eq:inequality3} to arrive at the bound 
\begin{equation*}
|\mathcal{F}[\mathbbm{1}_{G_{+1}} - \mathbbm{1}_{G_{-1}}](t)| \leq 2e^{-c_1(d-1)m^{\tau}}\sum_{j=0}^{\infty}\dfrac{(L|t_d|B^{d-1}/\gamma^2_m)^{2j+1}}{|t_d|(2j+1)!}.
\end{equation*}
Note that $ \sum_{j=0}^{\infty}\dfrac{(L|t_d|B^{d-1}/\gamma^2_m)^{2j+1}}{|t_d|(2j+1)!} $ is further bounded by
\begin{equation*}
LB^{d-1}(1/\gamma^2_m)\sinh(L|t_d|B^{d-1}/\gamma^2_m)
\end{equation*}
since
\begin{align*}
\sum_{j=0}^{\infty}\dfrac{(L|t_d|B^{d-1}/\gamma^2_m)^{2j+1}}{|t_d|(2j+1)!}
& = LB^{d-1}(1/\gamma^2_m)\sum_{j=0}^{\infty}\dfrac{(L|t_d|B^{d-1}/\gamma^2_m)^{2j}}{(2j+1)!}  \\
& \leq
LB^{d-1}(1/\gamma^2_m)\sum_{j=0}^{\infty}\dfrac{(L|t_d|B^{d-1}/\gamma^2_m)^{2j}}{(2j)!} \\
& = LB^{d-1}(1/\gamma^2_m)\sinh(L|t_d|B^{d-1}/\gamma^2_m).
\end{align*}
The last term is bounded by a constant since $ |t_d| \leq am^{\tau} = O(\gamma^2_m) $.
\end{proof}

\begin{lemma}\label{psi_decay}
Let $ \{\psi_j\} $ be a sequence of real-valued functions on $ \mathbb{R} $. Suppose there exists positive constants $ C > 0 $ and $ c > 0 $ such that 
\begin{equation*}
|\mathcal{F}[\psi_j](t)| \leq Ce^{-c|t|^{\tau}}, \quad j = 1, 2, \dots
\end{equation*}
for all $ t \in \mathbb{R} $,
where $ \tau \in (0,1] $. Then for all $ t \in \mathbb{R} $,
\begin{equation} \label{eq:Fourierconv}
|\mathcal{F}[\prod_{1 \leq j\leq k}\psi_j](t)| \leq C^kB^{k-1}e^{-c|t|^{\tau}/2}, \quad k = 1, 2, \dots,
\end{equation}
where $ B = \int_{\mathbb{R}}e^{-c|s|^{\tau}/2}ds $.
\end{lemma}

\begin{proof}[Proof of Lemma \ref{psi_decay}]
We will proof the claim using induction. To this end, suppose \eqref{eq:Fourierconv} holds. Then, using the fact that the Fourier transform of a product is the convolution of the individual Fourier transforms, we have
\begin{align*}
|\mathcal{F}[\prod_{1 \leq j\leq k+1}\psi_j](t)| 
& = |\mathcal{F}[\prod_{1 \leq j\leq k}\psi_j]\ast \mathcal{F}[\psi_{k+1}](t)| \\
& = \left|\int_{\mathbb{R}}\mathcal{F}[\prod_{1 \leq j\leq k}\psi_j](s)\mathcal{F}[\psi_{k+1}](t-s)ds\right| \\
& \leq \int_{\mathbb{R}}|\mathcal{F}[\prod_{1 \leq j\leq k}\psi_j](s)\mathcal{F}[\psi_{k+1}](t-s)|ds \\
& \leq C^{k+1}B^{k-1}\int_{\mathbb{R}}e^{-c|s|^{\tau}/2 - c|t-s|^{\tau}}ds.
\end{align*}
Next, note that the mapping $ x \mapsto |x|^{\tau} $ is H\"older continuous in the sense that 
\begin{equation*}
||x|^{\tau}-|y|^{\tau}| \leq |x-y|^{\tau},
\end{equation*}
for all $ x,y $ in $ \mathbb{R} $. Using this, we have that
\begin{equation*}
\int_{\mathbb{R}}e^{-c|s|^{\tau}/2 - c|t-s|^{\tau}}ds \leq e^{-c|t|^{\tau}/2}\int_{\mathbb{R}}e^{-c|s|^{\tau}/2}ds = Be^{-c|t|^{\tau}/2}.
\end{equation*}
Thus we have shown that
\begin{equation*}
|\mathcal{F}[\prod_{1 \leq j\leq k+1}\psi_j](t)| \leq C^kB^{k-1}e^{-c|t|^{\tau}/2}.
\end{equation*}
\end{proof}

\begin{lemma}\label{psi_exist}
Let $ a_1 \geq a_2 \geq \dots $ be a positive sequence with $ \sum_{j=1}^{\infty}a_j = 1 $. There exists a non-negative function $ \psi $ defined on $\R$ that is symmetric (i.e., $\psi(-x)=x$), infinitely many times differentiable, integrates to one (i.e., $ \int_{\mathbb{R}}\psi = 1 $), support equal to $ (-1/2, 1/2) $, and such that
\begin{equation} \label{eq:psi-derivatives}
\sup_{x\in [-1/2, 1/2]}\left|\frac{\diff^k\psi}{\diff x^k}(x)\right| \leq \frac{2^k}{a_1\ldots a_k}, \quad k = 1,2,\dots.
\end{equation}
In particular, for $ \tau \in (0,1) $ and $ a_j = \tfrac{1}{aj^{1/\tau}} $, where $ a = \sum_{j=1}^{\infty}\tfrac{1}{j^{1/\tau}} $, the function $\psi$ satisfies
\begin{equation*}
|\mathcal{F}[\psi](t)| \leq \exp\left\{-\tfrac{1}{e
\tau}\left(\tfrac{|t|}{2a}\right)^{\tau}\right\}, \quad \forall t\in\R.
\end{equation*}
Furthermore, $ \|\psi\|_{\infty} \leq 1 $, $ \|\psi^{\prime}\|_{\infty} \leq 2/(1-\tau) $, and $ \|\psi^{\prime}\|_{\infty} \leq 8/(1-\tau)^2 $.
\end{lemma}

\begin{proof}[Proof of Lemma \ref{psi_exist}]
The existence of $ \psi $ can be found in Theorem 1.3.5 of \cite{Hormander1995}. For the second conclusion, note that the identity
\begin{equation*} (-it)^k\mathcal{F}[\psi](t) = \int_{-1/2}^{1/2}e^{itx} \frac{\diff^k\psi}{\diff x^k}(x)dx, \quad k = 1,2,\dots 
\end{equation*}
holds. 
Using this and the upper bound for $\DS \frac{\diff^k\psi}{\diff x^k}$, we see that
\begin{equation*}
|t|^k|\mathcal{F}[\psi](t)| \leq (2a)^k(k!)^{1/\tau}.
\end{equation*}
Next, use the fact that $ k! \leq e^{k\ln k} $ to upper bound $ (2a)^k(k!)^{1/\tau} $ by $ \exp\{k\ln(2a)+(1/\tau)k\ln k\} $.
We have thus shown that 
\begin{equation*}
|\mathcal{F}[\psi](t)| \leq \exp\{k\ln(2a)+(1/\tau)k\ln k\}/|t|^k,
\end{equation*}
for $ t \neq 0 $ and $ k = 1,2,\dots $. Choose $ k = \tfrac{1}{e}\left(\tfrac{|t|}{2a}\right)^{\tau} $ so that
\begin{equation*}
|\mathcal{F}[\psi](t)| \leq \exp\left\{-\tfrac{1}{e
\tau}\left(\tfrac{|t|}{2a}\right)^{\tau}\right\}.
\end{equation*}
The estimates on the $ L_{\infty} $ norms of $ \psi $, $ \psi^{\prime} $, and $ \psi^{\prime\prime} $ follow from the fact that $ a \leq 1/(1-\tau) $.
\end{proof}

\begin{lemma} \label{lmm:smooth}
If $ \max_{x\in[-\delta, \delta]}g^{\prime\prime}(x) < 0 $, there exists $ L > 0 $, depending only on $ \tau $ and $ \gamma_m $, such that the sets $ G_{\omega} $ are convex.
\end{lemma}

\begin{proof}
As discussed in the proof of Theorem \ref{thm:lower}, the sets $ G_{\omega} $ are convex if the Hessian of $ b_{\omega} $ is negative-semidefinite. This is equivalent to showing that the largest eigenvalue of $ \nabla ^2 b_{\omega} $ is nonpositive. We can bound the maximum eigenvalue of $ \nabla ^2 b_{\omega} $ via
\begin{align*}
\lambda_{\text{max}}
& = \max_{\|u\|_2 = 1} u^{\prime}\nabla ^2 b_{\omega}u \\ 
& = \max_{\|u\|_2 = 1}[\sum_{k}g^{\prime\prime}(x_k)u^2_k + \sum_{k\ell}\omega(L/\gamma_m^{2})\frac{\partial^2 H_m}{\partial x_k\partial x_{\ell}}(x_1,\dots,x_{d-1})u_ku_{\ell}] \\
& \leq \max_{x\in[-\delta, \delta]}g^{\prime\prime}(x) + (L/\gamma_m^{2})\max\{ \|h_m\|_{\infty}^{d-3}\|h^{\prime}_m\|_{\infty}^2, \|h_m\|_{\infty}^{d-2}\|h^{\prime\prime}_m\|_{\infty} \} \\
& \leq \max_{x\in[-\delta, \delta]}g^{\prime\prime}(x) + (L/\gamma_m^{2})\max\{\|h^{\prime}_m\|_{\infty}^2, \|h^{\prime\prime}_m\|_{\infty} \}
\end{align*}

Now, from Lemma \ref{psi_exist} we have the estimates $ \|\psi\|_{\infty} \leq 1 $, $ \|\psi^{\prime}\|_{\infty} \leq 2/(1-\tau) $, and $ \|\psi^{\prime\prime}\|_{\infty} \leq 8/(1-\tau)^2 $. Thus,
\begin{align*}
|h^{\prime}_m(x)|
& = |\psi^{\prime}(x)\sin(\gamma_m x) - \gamma_m\psi(x)\sin(\gamma_m x)| \\
& \leq 2/(1-\tau) + \gamma_m,
\end{align*}
and
\begin{align*}
|h^{\prime\prime}_m(x)|
& = |\psi^{\prime\prime}(x)\cos(\gamma_m x) -2\gamma_m\psi^{\prime}(x)\sin(\gamma_m x)  - \gamma_m^2\psi(x)\cos(\gamma_m x)| \\
& \leq 8/(1-\tau)^2 + 4\gamma_m/(1-\tau) + \gamma^2_m.
\end{align*}
It thus follows that
\begin{equation*}
\max\{\|h^{\prime}_m\|_{\infty}^2, \|h^{\prime\prime}_m\|_{\infty} \} \leq 8/(1-\tau)^2 + 4\gamma_m/(1-\tau) + \gamma^2_m.
\end{equation*}
Next, choose $ L $, depending only on $ \tau $ and $ \gamma_m $, such that  
\begin{equation*}
(L/\gamma^2_m)[8/(1-\tau)^2 + 4\gamma_m/(1-\tau) + \gamma^2_m] \leq -(1/2)\max_{x\in[-\delta, \delta]}g^{\prime\prime}(x).
\end{equation*}
This means that $ \lambda_{\text{max}} \leq (1/2)\max_{x\in[-\delta, \delta]}g^{\prime\prime}(x) < 0 $.

\end{proof}

\bibliographystyle{plain}
\bibliography{Biblio}

\end{document}